\documentclass[11pt, letterpaper, oneside]{amsart}

\usepackage{latexsym, amsmath, amssymb, amsfonts, amscd,bm}
\usepackage{amsthm}
\usepackage{t1enc}
\usepackage[mathscr]{eucal}
\usepackage{graphicx, pb-diagram}
\usepackage{fancyhdr}
\usepackage{fancybox}
\usepackage{enumerate}
\usepackage{color}
\usepackage[all]{xy}
\usepackage{hyperref}
\usepackage{url}
\usepackage{tikz}
\usepackage[colorinlistoftodos,prependcaption,textsize=tiny,linecolor=none,backgroundcolor=white!25,bordercolor=blue]{todonotes}
\theoremstyle{plain}
\newtheorem{thm}{Theorem}[section]

\newtheorem{prop}[thm]{Proposition}
\newtheorem{lemma}[thm]{Lemma}
\newtheorem{cor}[thm]{Corollary}

\theoremstyle{definition}
\newtheorem{defi}[thm]{Definition}

\newtheorem{remark}[thm]{Remark}
\newtheorem{ex}[thm]{Example}

\newcommand{\RR}{\ensuremath{\mathbb R}}
\newcommand{\g}{\ensuremath{\mathfrak{g}}}

\definecolor{forest}{rgb}{0,0.5,0}

\begin{document}

\title{On the extension problem for weak moment maps}

\author{Leyli Mammadova}

\address{L. Mammadova, KU Leuven, Department of Mathematics, Celestijnenlaan 200B box 2400, BE-3001 Leuven,
Belgium.
\newline }\email{leyli.mammadova@kuleuven.be}
  
 \author{Leonid Ryvkin}
\address{L. Ryvkin, IMJ-PRG, Universit\'e Paris Diderot, B\^atiment Sophie Germain, 8 Place Aur\'elie Nemours, 75013 Paris, France.\newline
Also affiliated to: Faculty of Mathematics, Universit\"at Duisburg-Essen, Essen, Germany.\newline}
\email{leonid.ryvkin@imj-prg.fr}

\date{\today}
 
\begin{abstract}
	We compare  existence and equivariance phenomena for weak moment maps and homotopy moment maps in multisymplectic geometry.
\end{abstract}

\maketitle
\date{today}

\setcounter{tocdepth}{1} 
\tableofcontents

\section*{Introduction}
Multisymplectic (or $n$-plectic) geometry is a generalization of symplectic geometry, where $n=1$ corresponds to the symplectic case. Just as symplectic geometry has its origins in classical mechanics, multisymplectic geometry arose from classical field theory in the works of W. M. Tulczyjew, J. Kijowski and W. Szczyrba (see \cite{kijowski1973finite} and \cite{kijowski1975multisymplectic}). In \cite{cantrijn1999geometry}, F. Cantrijn,
A. Ibort, and M. de Le\'{o}n investigated multisymplectic manifolds from purely geometric viewpoint. Since then, there have been multiple attempts at generalizing the notion of a moment map from symplectic to multisymplectic geometry. Moment maps are an important concept in symplectic geometry: they formalize the Noether principle and have applications like the Marsden-Weinstein reduction, the Atiyah-Guillemin-Sternberg convexity theorem, and classification of toric manifolds.\\
In this paper, we compare two generalizations of symplectic moment maps to $n$-plectic geometry. 
One of them is the \emph{homotopy moment map} introduced in \cite{FRZ} by M. Callies, Y. Fregier, C. L. Rogers and M. Zambon. This map is an $L_{\infty}$-morphisms between the Lie algebra $\g$ acting on $(M,\omega)$ and the $L_{\infty}$-algebra $L_{\infty}(M,\omega)$ associated to $(M,\omega)$. The other map is the \emph{weak moment map} introduced in \cite{MR3815227, herman2018noether} by J. Herman, extending the notion introduced by T. B. Madsen and A. Swann in \cite{zbMATH06016606} and \cite{madsen2013closed}. Components of this map, restricted to  certain subspaces of $\Lambda^k \g$, are required to satisfy the first $n$ equalities satisfied by the components of a homotopy moment map. Hence, a homotopy moment map induces a weak moment map. In this paper, we answer some questions posed in \cite{MR3815227}, including the question of the reverse implication.\\
This paper is organized as follows. 
We begin with background material on multisymplectic geometry, Lie group/algebra actions and Cartan calculus in Section 1. In Section 2, we investigate the question of whether the existence of a weak moment map implies the existence of a homotopy moment map. This question is answered in Theorem 2.8, which is the main result of Section 2. Section 3 addresses questions of existence of equivariant homotopy (weak) moment maps using the framework of Section 2 and reproves a result from \cite{MR3815227}  extending it to the case of homotopy moment maps.

\bigskip
\paragraph{\bf Acknowledgements:} We thank Marco Zambon for helpful conversations and comments. L. M. acknowledges the support of the long term structural funding -- Methusalem grant of the Flemish Government. L. R. was supported by the Ruhr University Research School PLUS, funded by Germany's Excellence Initiative [DFG GSC 98/3] and the PRIME programme of the German Academic
Exchange Service (DAAD) with funds from the German Federal Ministry of Education and Research
(BMBF).

\section{Background}
In this section, we recall the relevant notions and results from multisymplectic geometry, Lie group/algebra actions and Cartan calculus.

\subsection{Lie group and Lie algebra actions} Let $G$ be a Lie group acting on a manifold from the left. 
\begin{defi} The vector field  \[v_x|_m=\frac{d}{dt}exp(-tx)\cdot m|
	_{t=0}, \forall m\in M \]
is called the \emph{infinitesimal generator} of the action corresponding to $x\in \g$.
\end{defi}

\begin{defi} Let $v_{x_i}$ be the infinitesimal generator corresponding to $x_i\in \g.$ Then for $p:=x_1\wedge...\wedge x_k\in \Lambda^k\g$, we call  $v_p:=v_{x_1}\wedge...\wedge v_{x_k}$ \emph{the infinitesimal generator of the action} corresponding to $p.$
\end{defi}

Next, we introduce the Lie algebra homology differential, which we will need in particular to define homotopy moment maps.
\begin{defi}
	Let $\g$ be a Lie algebra. The map $\delta_k:\Lambda^k\g \to \Lambda^{k-1}\g$ defined by
	\begin{equation}\label{def:homdiff}
	\delta_k: x_1\wedge ...\wedge x_k \mapsto \sum\limits_{1\leq i<j\leq k}(-1)^{i+j}[x_i,x_j]\wedge x_1\wedge ...\widehat{x_{i}}\wedge ...\wedge \widehat{x_j}\wedge...x_k,
	\end{equation}
	where $k\geq 1$ and $x_i \in \g$, is called \emph{k-th Lie algebra homology differential} of $\g$. 
	
\end{defi}
We recall the following definition from \cite{madsen2013closed}.
\begin{defi}
	The kernel of $\delta_k$ is called \emph{the k-th Lie kernel} of $\g$.
\end{defi}

We will denote the $k$-th Lie kernel of $\g$ by  $P_{k,\g}:=ker\delta_k$ and the direct sum of all the Lie kernels by $P_{\g}:=\bigoplus^{dim\hspace{0.02cm}\g}_{k=0}P_{k,\g}$. We will also be interested in the subspace $P_{\geq 1,\g}:=\bigoplus^{dim\hspace{0.02cm}\g}_{k=1}P_{k,\g}$.

\begin{remark}\label{morphism}
	Let $\g$ act on manifold $M$, and let $\bar{\delta_k}:\Gamma(\Lambda^k(TM))\to \Gamma(\Lambda^{k-1}(TM))$ be defined analogously to \eqref{def:homdiff}. Then, as $\mathfrak g\to \mathfrak X(M)$ is a Lie algebra homomorphism, we have $\bar{\delta_k} v_p=v_{\delta_k p}$  for $p\in \Lambda^k\g$. In particular, for $p\in P_{\g}$, $\bar{\delta_k}v_p=v_{\delta_k p}=0$.
\end{remark}

The following formula is known as the "Extended Cartan Formula" and can be found in \cite[Lemma 3.4]{madsen2013closed}.

\begin{lemma} \label{cartan}
	Let $\alpha\in \Omega^m(M)$. Then for all $k\geq 2$ and all vector fields $v_1,...,v_k,$ we have:
	\begin{align*}
	(-1)^kd\iota(v_1\wedge...\wedge v_k)\alpha = & \iota(\bar{\delta_k}(v_1\wedge...\wedge v_k))\alpha\\
	& +\sum_{i=1}^{k}(-1)^i\iota(v_1\wedge...\wedge\hat{v_i}...\wedge v_k)\pounds_{v_i}\alpha\\
	& +\iota(v_1\wedge...\wedge v_k)d\alpha.
	\end{align*}
\end{lemma}

\subsection{Lie algebra cohomology}

We briefly recall some notions from the realm of Lie algebra cohomology with values in a module. For a more detailed and systematic exposition, we refer the reader to \cite{MR1269324}.\\
Let $\g$ be a Lie algebra and $\mathcal M$ a $\g$-module, i.e. a vector space equipped with a Lie algebra homomorphism map $a:\g\to End(\mathcal M),~a\mapsto a_x$. When no confusion is possible, we will write $x\cdot m$ for $a_x(m)$. We equip the space $\Lambda\mathfrak g^*\otimes \mathcal M$ with a differential $d_{\g,\mathcal M}$ as follows. For $\alpha \in \Lambda^k\mathfrak g^*\otimes \mathcal M$, $d^k_{\g,\mathcal M}(\alpha)=d_{\g, \mathcal M}(\alpha)\in \Lambda^{k+1}\mathfrak g^*\otimes \mathcal M$ is given by
\begin{align*}
d_{\g,\mathcal M}(\alpha)(x_1\wedge ...\wedge x_{k+1})=\alpha(\delta(x_1\wedge ...\wedge x_{k+1}))+\sum_{i=1}^{k+1}(-1)^{i+1}x_i\cdot \alpha(x_1\wedge ...\wedge \hat x_i\wedge ...\wedge x_{k+1}).
\end{align*}
This differential squares to zero, so we can define its cohomology groups:

\begin{defi}
	We define the \emph{$k$-th Lie algebra cohomology group of $\g$ with values in $\mathcal M$} as
	\[
	H^k(\g,\mathcal M)=\frac{ker(d^k_{\g,\mathcal M})}{Im(d^{k-1}_{\g,\mathcal M})}.
	\]
\end{defi}

The following two special cases will be of particular interest in the sequel.
\begin{ex}
	In case $\mathcal M=\mathbb R$ is the trivial $\g$-module $d_\g:=d_{\g, \mathbb R}$ is the dual of $\delta$ and we recover Chevalley-Eilenberg cohomology $H^k(\mathfrak g, \mathbb R)=H^k(\mathfrak g)$.
\end{ex}

\begin{ex}
	The group $H^0(\g, \mathcal M)=ker(d_{\g,\mathcal M}^1)\subset \mathcal M$ is just the subspace of $\mathcal M^\g$ of $\g$-invariant elements in $\mathcal M$.
\end{ex}

\subsection{Multisymplectic geometry} In this subsection, we recall the relevant notions from multisymplectic geometry. 
We begin with the definition of a multisymplectic manifold.
\begin{defi}
	A pair $(M,\omega)$ is an \textit{n}\emph{-plectic} manifold, if  $\omega$ is a closed nondegenerate $n+1$ form, i.e., \[d\omega=0\] 
	
	\noindent and the map $\iota{\_}\omega:TM\to \Lambda^{n}T^{\ast}M, v\mapsto \iota_{v}\omega$
	is injective.\end{defi}
Thus, a symplectic manifold is a 1-plectic manifold. As in symplectic geometry, in $n$-plectic geometry there is also a notion of Hamiltonian vector fields.
\begin{defi}
	An $(n-1)$-form $\alpha$ on an $n$-plectic manifold $(M,\omega)$ is \emph{Hamiltonian} if there exists a vector field $v_{\alpha}\in\mathfrak{X}(M)$ such that \[d\alpha=-\iota_{v_{\alpha}}\omega.\]
	The vector field $v_{\alpha}$ is the \emph{Hamiltonian vector field} corresponding to $\alpha$.
\end{defi}

The space of smooth functions $C^{\infty}(M)$ on a symplectic manifold forms a Lie (Poisson) algebra. An $n$-plectic manifold is canonically equipped with a Lie $n$-algebra $L_{\infty}(M,\omega)$ defined in the following way \cite[Thm. 5.2]{RogersL}:

\begin{defi}
	Given an $n$-plectic manifold, there is a corresponding Lie $n$-algebra $L_{\infty}(M,\omega)=(L,\{[\;,\;...\;,\; ]'_k\})$ with underlying graded vector space 
	\[
	L_i=\begin{cases}
	\Omega_{\textrm{Ham}}^{n-1}(M)\hspace{0.5cm}i=0\\
	\Omega^{n-1+i}(M)\hspace{0.2cm}1-n\leq i<0
	\end{cases}
	\]
	and maps $\{[\;,\;...\;,\; ]'_k:L^{\otimes k}\to L\;|\;1\leq k<\infty\} $ defined as 
	\[[\alpha]'_1=d\alpha,\hspace{0,2cm}\textrm{if }|\alpha|<0\]
	and, for $k>1$, \[
	[\alpha_{1},\;...\;,\alpha_{k} ]'_k=\begin{cases}
	\zeta(k)\iota(v_{\alpha_{1}}\wedge...\wedge v_{\alpha_{k}})\omega \hspace{0.2cm}\textrm{if }|\alpha_{1}\otimes...\otimes\alpha_{k}|=0\\
	0\hspace{3.7cm}\textrm{if }|\alpha_{1}\otimes...\otimes\alpha_{k}|<0,
	\end{cases}
	\]
	where $v_{{\alpha}_i}$ is any Hamiltonian vector field associated to $\alpha_i\in{\Omega}_{\textrm{Ham}}^{n-1}(M)$, $\zeta(k)=-(-1)^{\frac{k(k+1)}{2}}$, 
	and $\iota(\dots)$ denotes contraction with a multivector field.
\end{defi}

\begin{remark}
A \emph{Lie $n$-algebra} is an $L_{\infty}$-algebra concentrated in degrees $0,-1,...,1-n$. 
For definition and properties of $L_{\infty}$-algebras see \cite{lada1993introduction}.
\end{remark}
\begin{remark}
For $n=1$ in the above definition, we get the Lie algebra $C^{\infty}(M)$ associated to a symplectic manifold $(M,\omega)$.
\end{remark}

\subsection{Homotopy moment maps and weak moment maps} In this subsection, we  introduce the main characters of this paper: the homotopy moment map and the weak moment map.

Following \cite{FRZ} we recall:
\begin{defi} \label{def:gmomap}
	Let {$\g\to \mathfrak{X}(M), {x\mapsto v_x}$} be a Lie algebra action on an $n$-plectic manifold $(M,\omega)$ by Hamiltonian vector fields. 
	A \emph{homotopy moment map} for this action is  an $L_{\infty}$-morphism 
	\begin{equation*}
	(\widetilde f_k):\g\to L_{\infty}(M,\omega)
	\end{equation*} 
	such that \begin{equation*}-\iota_{v_{x}}\omega=d(\widetilde{ f}_1(x))\hspace{0.1cm},\forall x\in\g.\end{equation*}

	This is equivalent to a collection of $n$ linear maps 
	\begin{align*}
	\widetilde f_1&:\g\to \Omega_{\textrm{Ham}}^{n-1}(M)\\
	\widetilde f_k&:\Lambda^k \g\to \Omega^{n-k}(M), \hspace{0.1cm}2\leq k\leq n
	\end{align*}
	such that 
	\begin{equation} \label{homotopymomap}
	-\widetilde f_{k-1}(\delta p)=d \widetilde f_k(p)+\zeta(k)\iota_{v_p}\omega
	\end{equation}
	
	for $1\leq k \leq n+1$, where $\widetilde f_0, \widetilde f_{n+1}\equiv 0$, and $v_p$ is the infinitesimal generator of the action corresponding to $p\in \Lambda^k\g$. 
	
\end{defi}
   
The next definition is taken from \cite[Def.3.11]{herman2018noether}.
\begin{defi} Let {$\g\to \mathfrak{X}(M), {x\mapsto v_x}$} be a Lie algebra action on an $n$-plectic manifold $(M,\omega)$ by Hamiltonian vector fields. 
A \emph{weak moment map} is a collection of linear maps $\widehat f_k:P_{k,\g}\to \Omega^{n-k}(M)$, where $1\leq k \leq n$, satisfying \[d\widehat f_k(p)=-\zeta(k)\iota_{v_p}\omega \]
for $k\in 1,...,n$ and all $p\in P_{k,\g}$.
\end{defi} 
\begin{remark}
We stress that we have $n$ equations in the definition of a weak moment map, while a homotopy moment map has to obey $n+1$ equations. 
\end{remark} 

\section{Existence}

Every homotopy moment map induces a weak moment map by restriction to the Lie kernel, i.e., if there is a homotopy moment map for a given Lie algebra action, there is also a weak moment map.
In this section, we will show that under an additional assumption the converse is also true. Throughout this section, $(M,\omega)$ is an $n$-plectic manifold, and $\g$ acts on $(M,\omega)$  preserving $\omega$. 

\subsection{Existence of homotopy moment maps}
 We recall the existence result for homotopy moment maps from \cite{FLRZ} and \cite{LeonidTillmannComoment}. We consider the double complex 
 \begin{equation} \label{big complex}
 (\Lambda^{\geq 1}\g^{*}\otimes\Omega(M),d_{\g},d)
 \end{equation}
 where $d_{\g}$ is the Chevalley-Eilenberg differential of $\g$, and $d$ is the de Rham differential of $M$. Let $\widetilde C$ be the total complex with differential 
 \begin{equation} \label{dtot}
 \widetilde d_{tot}:=d_{\g}\otimes1+1\otimes d
 \end{equation}
 
 Define \begin{align*}
 \omega_k: \Lambda^k \mathfrak g&\to \Omega^{n+1-k}(M)\\
 p&\mapsto \iota_{v_p}\omega
 \end{align*} 
 and 
 \[\widetilde{\omega}:=\sum_{k=1}^{n+1}(-1)^{k-1}\omega_k \]
 
 We recall the following result from \cite{FLRZ} and \cite{LeonidTillmannComoment}.
 
 \begin{prop} \label{characterizationhomotopymomap}\label{prop:hotexists}
 	Let $\widetilde\alpha:=\widetilde\alpha_1+...\widetilde\alpha$, with $\widetilde\alpha_k \in \Lambda^k\g^{*}\otimes \Omega^{n-k}(M).$ Then $\widetilde d_{tot}\widetilde\alpha=\widetilde{\omega}$ if and only if
 	\[\widetilde  f_k:=\zeta(k)\widetilde\alpha_k:\Lambda^k\mathfrak g\to\Omega^{n-k}(M),~~~	k=1,...,n, \]
 are the components of a homotopy moment map for the action of $G$ on $(M,\omega).$
 \end{prop}
 In other words, a homotopy moment map for the action of $\g$ on $(M,\omega)$ exists if and only if the class $[\widetilde{\omega}]=0$ in $H^{n+1}(\widetilde C)$. \\
 
 By the Kuenneth theorem  $H^{n+1}(\widetilde C)=\bigoplus_{i=1}^{n+1} H^i(\mathfrak g)\otimes H^{n+1-i}(M)$, where $H^i(\mathfrak g)$ are the Lie algebra cohomology groups and $H^{n+1-i}(M)$ the de Rham cohomology groups. In particular, the class $[\widetilde \omega]$ can be decomposed into individual classes in $H^i(\mathfrak g)\otimes H^{n+1-i}(M)$. The case of $i=n+1$ will be of particular interest to us and can be described in the following non-cohomological manner:

\begin{lemma} \label{condition}
	If there exists a homotopy moment map for $\g$ acting on $(M,\omega)$, then the map
	\begin{align*}
	\phi:P_{n+1,\g} &\to C^{\infty}(M)\\
	p&\mapsto \iota_{v_p}\omega
	\end{align*}
	vanishes identically.
\end{lemma}

\begin{proof}
	If there exists a homotopy moment map $f$ for the action of $\g$ on $(M,\omega)$, then it has to satisfy equation \eqref{homotopymomap} for $k=n+1,$ i.e.,
	\[-\widetilde f_{n}(\delta p)=\zeta(n+1)\iota_{v_p}\omega\]
	for all $p \in \Lambda ^{n+1}\g.$ This means that for $p \in P_{n+1,\g}$, we have $\iota_{v_p}\omega=0$.
\end{proof}

\begin{remark} \label{classc}
	In \cite[Cor 9.3]{FRZ}, the following map was introduced:
	
	\begin{equation} \label{c}
	c_m^{\g}: \Lambda^{n+1}\g \to \RR,~~
	p\mapsto (-1)^n\zeta(n+1)\iota_{v_p}\omega|_{m},
	\end{equation}
	for some $m\in M$. It was shown in \cite{FRZ} that this map is a $(n+1)$-cocycle in $\Lambda \g^*$, and, for connected $M$, the cohomology class $[c_m^{\g}]\in H^{n+1}(\g)$ does not depend on the point $m\in M$. Moreover, \cite[Prop. 9.5]{FRZ} states that if a connected $n$-plectic manifold $(M,\omega)$ is equipped with a $G$-action which induces a homotopy moment map, then
	\begin{equation*}
	[c_m^{\g}]=0.
	\end{equation*}
	Up to sign, the class $[c^\g_m]$ can be interpreted as the evaluation at $m$ of the $H^{n+1}(\mathfrak g)\otimes H^0(M)$ component of $[\widetilde{\omega}]$.	
	
\end{remark}

\begin{lemma} \label{phiandc} Let $(M,\omega)$ be a connected $n$-plectic manifold equipped with a $G$-action preserving $\omega$. Let $c_m^{\g}$ be defined as in equation \eqref{c}, and $\phi$ be defined as in Lemma \ref{condition}.
Then the condition \[[c_m^{\g}]=0\] is equivalent to $\phi\equiv 0$.
\end{lemma}
\begin{proof}
Indeed, assume $\phi\equiv 0$. Then, for $p\in P_{n+1,\g}$
\begin{align*}
c^{\g}_m(p)=(-1)^n\zeta(n+1)(\iota_{v_p}\omega)|_{m}=0
\end{align*}
for any $p\in M$.
That means that $c^{\g}_m\in (P_{n+1,\g})^\circ$, where $(P_{n+1,\g})^\circ$ is the annihilator of $P_{n+1,\g}$. Therefore $c^{\g}_m\in (P_{n+1,\g})^\circ=(ker \hspace{0.05 cm}\delta_{n+1})^\circ=im \hspace{0.05 cm} d^{n}_{\g},$
i.e., $[c_m^{\g}]=0$.\\

Conversely, suppose $[c_m^{\g}]=0$. Since $M$ is connected, this holds at all $m\in M$. Then $c_m^{\g}=d_{\g}\xi$
for some $\xi\in \Lambda^n\g^*$. For any $p\in P_{n+1,\g}$ and $m\in M$ this means
\begin{align*}
\iota_{v_p}\omega|_{m}=	(-1)^n\zeta(n+1)c_m^{\g}(p)= (-1)^n\zeta(n+1)(d_\g\xi)(p)=(-1)^n\zeta(n+1)\xi(\delta p) =0.
\end{align*}
\end{proof}
 
\subsection{Existence of weak moment maps} 
  We will need the following lemma.
  
  \begin{lemma}\label{closed} 
  	For $p\in P_{k,\g}, k=1,...,n+1,$ the form $\iota_{v_p}\omega$ is closed, i.e., $$d(\iota_{v_p}\omega)=0.$$
  	In particular, $p\mapsto \iota_{v_p}\omega$ induce well-defined maps  $ P_{k,\g}\to  H^{n+1-k}(M)$.
  \end{lemma}
  \begin{proof}
  	This follows from Remark \ref{morphism} and Lemma \ref{cartan}.
  \end{proof}
  
  \begin{lemma}\label{vanish}
   A weak moment map exists if and only if the maps
  	\begin{align*}
  	P_{k,\g} &\to H^{n+1-k}(M)\\
  	p&\mapsto [\iota_{v_p}\omega]
  	\end{align*}
  	are identically zero for $1\leq k \leq n$.
  \end{lemma}
  \begin{proof}
  	By definition of weak moment maps, for all $p\in P_{k,\g}$ and all $k$, the form $\iota_{v_{p}}\omega$ has to be exact.
  \end{proof}
 
Analogously to Proposition \ref{characterizationhomotopymomap}, we can encode weak moment maps as primitives of a certain element in the double complex
\begin{equation} 
(P_{\geq 1, \mathfrak g}^{*}\otimes\Omega(M),0,d)
\end{equation}
with zero differential on $P_{\geq 1,\mathfrak g}^{*}$ and de Rham differential on $M$. Let $\widehat C$ be the total complex with differential $ \widehat{d}_{tot}:=1\otimes d$.
We define $\widehat{\omega}\in \widehat C$ as
\[\widehat{\omega}:=\sum_{k=1}^{n}(-1)^{k-1}(\omega_k|_{P_{k,\mathfrak g}}). \]
From the discussion above, it follows that:

\begin{prop}\label{prop:weakexist}
	Let $\widehat\alpha:=\widehat\alpha_1+...+\widehat\alpha_n$, with $\widehat\alpha_k \in P_{k,\g}^{*}\otimes \Omega^{n-k}(M).$ Then $\widehat d_{tot}\widehat\alpha=\widehat{\omega}$ if and only if
	\[\widehat f_k:=\zeta(k)\widehat\alpha_k: P_{k,\g}\to\Omega^{n-k}(M),~~~~k=1,...,n\]
	are the components of a weak moment map for the action of $G$ on $(M,\omega).$	I.e., the existence of a weak moment map is equivalent to  the vanishing of $[\widehat \omega]\in H^{n+1}(\widehat C)$.
\end{prop}

\subsection{Main result of the section}

\begin{thm} \label{existence}
Let $\g$ act on $(M,\omega)$ by preserving $\omega$. The following statements are equivalent:
\begin{enumerate}
	\item The action admits a homotopy moment map
	\item The action admits a weak moment map and $\phi\in  P^*_{n+1,\g} \otimes C^{\infty}(M) $ as defined in Lemma \ref{condition} vanishes identically.
\end{enumerate}

\end{thm} 

\begin{proof}
	
	The implication $(1)\to (2)$ is an immediate consequence of the fact that any homotopy moment map restricts to a weak moment map and Lemma \ref{condition}. To prove the converse, we first observe that $\widetilde \omega |_{P_{\geq 1,\mathfrak g}}= \widehat \omega + \phi$. In other words, the restriction 
	\begin{equation} \label{chainmap}
	(res\otimes id):\widetilde{C}=\Lambda^{\geq 1}\mathfrak g^*\otimes \Omega(M)\to{P^*_{\geq 1,\mathfrak g}}\otimes \Omega(M)=\widehat{C}
	\end{equation}
	
	 maps $\widetilde \omega$ to $\widehat{\omega} +\phi$. The restriction is a chain map, so to complete the proof we just have to verify that its induced map in cohomology $[res\otimes id]:H(\widetilde{ C})\to H(\widehat{ C})$ is injective. By using the naturality of the Kuenneth isomorphism we get the following commutative diagram:
	
	\begin{center}
		\begin{tikzpicture} 
		\node (a)  at (0,3) {$H^{n+1}(\Lambda^{\geq 1} \g^{*}\otimes \Omega(M))$};
		\node (b)  at (6,3) 
		{$H^{n+1}( P^{*}_{\g}\otimes \Omega(M))$};
		\node (c)  at (0,0) {$\bigoplus_{k\geq 1} H^k\g \otimes H^{n+1-k}(M)$};
		\node (d) at (6,0) {$\bigoplus_{k\geq 1} P^{*}_{k,\g}\otimes H^{n+1-k}(M)$};

		\path[->] (a) edge node[below] {$[res\otimes id]$} (b);
		\path[->] (b) edge node[right] {$\cong$} (d);
		\path[->] (c) edge node[above] {$[res] \otimes [id]$} (d);
		\path[->] (a) edge node[left] {$ \cong $} (c);
		
		\end{tikzpicture}
	\end{center}
	
	Thus, to prove our claim it suffices to verify that $[res]:H^k(\mathfrak g)\to P_{k,\g}^*$ is injective. To see this, consider the exact sequence \[0\to P_{k,\g} \xrightarrow{i} \Lambda^k\g \xrightarrow{\delta_k} \Lambda ^{k-1}\g.\] 
	Dualizing it leads to the exact sequence \[0 \leftarrow P_{k,\g}^{*} \xleftarrow{\pi} \Lambda^k{\g^{*}} \xleftarrow{d^{k-1}_{\g}} \Lambda^{k-1}{\g^{*}}\]	 
	 Therefore, 
	 \[P^*_{k,\g}=\Lambda^k{\g^{*}}/im(d^{k-1}_{\g})\hookleftarrow ker(d^k_{\g}) /im(d^{k-1}_{\g})= H^k(\g).\]
\end{proof}

\begin{cor}
A weak moment exists if and only if the projection of $[\widetilde \omega]$ to
$$ \bigoplus_{k=1}^nH^k(\mathfrak g)\otimes H^{n-k+1}(M)$$ vanishes.
\end{cor}
\begin{proof}

{
Let $p^k_i$, where $i=1,...,dim\hspace{0.1cm}P_{k,\g}$ be a basis of $P_{k,\g}$. Then $\omega_k|_{P_{k,\mathfrak g}}$ can be written as $\omega_k|_{P_{k,\mathfrak g}}=\sum\limits_{i} (p^k_i)^*\otimes \omega_k(p^k_i)$. Note that, since each $\omega_k(p^k_i)$ is closed by Lemma \ref{closed}, each $\omega_k|_{P_{k,\mathfrak g}}$ is also closed under $\widehat{d}_{tot}=1\otimes d$. Thus, the image of $[\widehat{\omega} + \phi]$ under the Kuenneth isomorphism $\varkappa$ is
\begin{align*}
\varkappa([\widehat{\omega} + \phi]):&=\sum\limits_{k=1}^{n}(-1)^{k-1}\varkappa[\omega_k|_{P_{k,\mathfrak g}}] +(-1)^n\varkappa[\phi]\\
&=\sum\limits_{k=1}^{n}(-1)^{k-1}\sum\limits_{i} (p^k_i)^*\otimes [\omega_k(p^k_i)] + (-1)^n\varkappa[\phi]\\
\end{align*} 
so $\varkappa([\widehat{\omega} + \phi])$ can be divided into two parts:
$$\sum\limits_{k=1}^{n}(-1)^{k-1}\varkappa[\omega_k|_{P_{k,\mathfrak g}}]\in \bigoplus_{k=1}^n P_{k,\g^*}\otimes H^{n-k+1}(M)$$
and 
$$(-1)^n\varkappa[\phi]\in P_{n+1, \g^*}\otimes H^0(M).$$}

{In particular, $$\varkappa([\widehat{\omega}])= \sum\limits_{k=1}^{n}(-1)^{k-1}\varkappa[\omega_k|_{P_{k,\mathfrak g}}]\in \bigoplus_{k=1}^n P_{k,\g^*}\otimes H^{n-k+1}(M).$$}

{Since
$[res]\otimes[id](\varkappa[\widetilde{\omega}])=\varkappa([\widehat{\omega}]) + \varkappa([\phi]),$
it is clear from the nature of the map $[res]\otimes [id]$ that the preimage of  $\varkappa([\widehat{\omega}])$ under $[res]\otimes[id]$ is the projection of $\varkappa([\widetilde{\omega}])$ to $\bigoplus_{k=1}^nH^k(\mathfrak g)\otimes H^{n-k+1}(M).$}
{Since the map $[res]\otimes[id]$ is injective, $\varkappa([\widehat{\omega}])$ vanishes if and only if that preimage vanishes. Noting that the Kuenneth map is an isomoprphism, we get the statement of the corollary.}

\end{proof}

With this corollary, we can recover the following results from \cite{MR3815227}:

\begin{prop}[5.12 in \cite{MR3815227}]
	If $H^1(\g)=...=H^n(\g)=0$, then a weak moment map exists.
\end{prop}

\begin{thm}[5.6 in \cite{MR3815227}]
	If $H^0(\g, P_{k,\g})=0$ for $k=1,...,n$, then a weak moment map exists, where the $\g$-module structure on $P_{k,\g}$ is induced by the adjoint action.
\end{thm}
\begin{proof}
	Recall that the zeroth Lie algebra cohomology $H^0(\g, \mathcal M)$ equals the subspace of $\g$ invariants $\mathcal M^\g$ in $\mathcal M$. As $H^k(\g)\subset P^*_{k,\g}$, we have  $H^k(\g)^\g\subset (P_{k,\g}^*)^\g$. But $H^k(\g)^\g=H^k(\g)$, so if $H^0(\g, P_{k,\g})=(P_{k,\g}^*)^\g=0$ for $k=1,...,n$, then also $H^k(\g)=0$ and a weak moment exists.
\end{proof}	

Below we provide two examples illustrating the necessity of requiring $\phi$ to vanish in the assumptions of Theorem \ref{existence}. The first example deals with the more familiar symplectic case:

\begin{ex} \label{symplecticexplicit}
Let $n=1$, i.e., consider a connected Lie group $G$ acting on a {connected} symplectic manifold $(M,\omega)$. 
{In this case, a weak moment map is a map $\widehat{ f}:\g\to C^{\infty}(M)$ that satisfies \[d\widehat{ f}(x)=-\iota_{v_x}\omega,\] 
and a homotopy moment map is a map 
$\widetilde{ f}:\g\to C^{\infty}(M)$ that satisfies
\begin{align*} \label{symplecticmomap}
d\widetilde{ f}(x)&=-\iota_{v_x}\omega\\
\widetilde{ f}([x,y])&=\{\widetilde{ f}(x),\widetilde{ f}(y)\}=\omega(v_x,v_y).
\end{align*}  In symplectic geometry, the latter map is called an \emph{equivariant moment map}.
It is then well-known (see, e.g., \cite[\S 26 ]{Ana}) that if there exists a weak moment map, then the obstruction to the existence of an equivariant moment map lies in $H^2(\g).$} 

{More specifically, let $\widehat{ f}$ be a weak moment map. Consider 
\begin{align*} 	
	h(x,y):&=\{\widehat{f}(x),\widehat{f}(y)\} - \widehat{ f}([x,y])\\
	 &=\omega(v_x,v_y) - \widehat{ f}([x,y]).
\end{align*}
Since $d(\omega(v_x,v_y))=-\iota_{[v_x,v_y]}(\omega)=d\widehat{ f}([x,y]),$ it follows that $h(x,y)$ is a constant function on $M$, and therefore it defines an element $h\in \Lambda^2\g^*$.
Evaluating $h({x,y})$ at any point $m\in M$, we get
\begin{align*}
h(x,y)&=\omega(v_x,v_y)|_{m} - \widehat{ f}|_{m}([x,y])\\
&=c^{\g}_m(x,y) + d_{\g}\widehat{f}(x,y)|_{m}.
\end{align*}
If we assume that $c^{\g}_m$ is exact, then $h\in \Lambda^2\g^*$ is exact, i.e., there exists $b\in \g^*$ such that $h=d_{\g}b.$
Then $\widetilde{ f}:=\widehat{ f}-b$ is an equivariant moment map. Indeed,
\begin{align*}
\widetilde{f}([x,y])-\omega(v_x,v_y)&=\widehat{f}([x,y])-b([x,y])-\omega(v_x,v_y)\\
&=\widehat{f}([x,y])+d_{\g}b(x,y)-\omega(v_x,v_y)\\
&=\widehat{f}([x,y])+h(x,y)-\omega(v_x,v_y)\\
&=0
\end{align*} 
Note that, by Lemma \ref{phiandc}, the exactness of $c^{\g}_m$ is equivalent to the vanishing of $\phi$. Thus, if there exists a weak moment map for a Lie group acting on a symplectic manifold, then there exists a homotopy moment map if and only if $\phi\equiv0$.} 
\end{ex}	

Thus, the example above yields many instances in symplectic geometry where a weak moment map exists, but a homotopy moment map does not. 
The next example illustrates such a case in 2-plectic geometry.

\begin{ex}[\cite{FRZ}]
Let $G$ be a connected compact semi-simple Lie group acting on itself by left multiplication. The Lie algebra $\g$ of $G$ is equipped with a non-degenerate skew-symmetric form defined by
\begin{equation*} \label{cartanform}
\theta(x,y,z):=\langle x,[y,z]\rangle,
\end{equation*}
where $\langle\_,\_\rangle$ is the Killing form. Let $\omega$ be the left-invariant form which equals $\theta$ at the identity element $e$. This form is closed and non-degenerate, i.e., $(G,\omega)$ is a 2-plectic manifold.
It is well-known that $H^1(\g)=H^2(\g)=0$ (see, e.g., \cite{chevalley1948cohomology}), so a weak homotopy moment map exists for this action. However, the map $\phi$ defined in Lemma \ref{condition} does not vanish, hence, there is no homotopy moment map.
\end{ex}

In the next example, we will {explicitly} construct a {2-plectic} homotopy moment map from a weak moment map.

\begin{ex} \label{so(3)}
	Consider the action of $SO(3)$ on $(M=\RR^3, \omega=dx_1\wedge dx_2\wedge dx_3)$ by rotations. Note that, since $H^1(M)=0$, the first component of a weak moment map for this action exists. 
	The Lie algebra $\mathfrak{so}(3)$ is spanned by the following elements:
	\vspace{0.2cm}
	\begin{center}
		$e_1=\begin{pmatrix}
		0 & -1 & 0\\
		1 & 0 & 0\\
		0 & 0 & 0
		\end{pmatrix},$ $e_2=\begin{pmatrix}
		0 & 0 & -1\\
		0 & 0 & 0\\
		1 & 0 & 0
		\end{pmatrix},$ $e_3=\begin{pmatrix}
		0 & 0 & 0\\
		0 & 0 & -1\\
		0 & 1 & 0
		\end{pmatrix}$. 
	\end{center}

	\vspace{0.2cm}
	These elements satisfy the following bracket relations: 
	
	\begin{center}
		$[e_1,e_2]=e_3$,$[e_1,e_3]=-e_2$,$[e_2,e_3]=e_1,$
	\end{center}
	
	therefore, $P_{2,\mathfrak{so}(3)}\equiv 0,\hspace{0.1cm} P_{3,\mathfrak{so}(3)}=\Lambda^3\mathfrak{so}(3)$, and a weak homotopy moment map for this action has only one component. Note further, that the orbits of this action are of dimensions 0 and 2, hence $\phi\in  P^*_{3,\mathfrak{so}(3)} \otimes C^{\infty}(M)=\Lambda^3\mathfrak{so}(3)^*\otimes C^{\infty}(M)$ defined in Lemma \ref{condition} vanishes. Therefore, by Theorem \ref{existence}, there exists a homotopy moment map for this action.
	
	The infinitesimal generators $v_i$ corresponding to elements $e_i$ are given by 
	\begin{center}
		$v_1=x_2\frac{\partial}{\partial x_1}-x_1\frac{\partial}{\partial x_2}$,
		$v_2=x_3\frac{\partial}{\partial x_1}-x_1\frac{\partial}{\partial x_3}$,
		$v_3=x_3\frac{\partial}{\partial x_2}-x_2\frac{\partial}{\partial x_3}.$
	\end{center}
	Defining 
	\begin{center} 
		$\widehat{f}_1(e_1)=\omega(v_2,v_3)$,
		$\widehat{f}_1(e_2)=-\omega(v_1,v_3)$,
		$\widehat{f}_1(e_3)=\omega(v_1,v_2)$
	\end{center} 
	on basis elements gives the weak moment map.
	
	Note that $\widetilde{ f}=(\widetilde{ f}_1=\widehat{f}_1, \widetilde{ f}_2\equiv 0)$ is a pre-image of $\widehat{f}$ under the map \eqref{chainmap}. To construct a homotopy moment map out of $\widehat{f}$, note that $\widetilde{\psi}:=\widetilde{\omega} -\widetilde{d}_{tot}\widetilde{ f}$ is $\widetilde{d}_{tot}$-exact and the sum of any primitive with $\widetilde{f}$ gives a homotopy moment map. So, to construct a homotopy moment map, we need to find a primitive of $\widetilde{\psi}$, i.e., find a $\widetilde{h}=(h_1,h_2)\in \widetilde{C}^2$ that satisfies the following equations:
	\begin{align*} 
	-dh_1&=0\\
	\nonumber d_{\g}h_1 +dh_2 &= -\omega_2 - d_{\g}\widehat{ f}_1\\
	\nonumber d_{\g}h_2&=\omega_3.
	\end{align*}
	Note that $\omega_3=0$. Also, evaluating $-\omega_2 - d_{\g}\widehat{ f}_1$ on basis elements, we get, using the definition of $\widehat{f}_1$ and bracket relations between the $e_i$, \[\omega_2(e_i,e_j) {-}f_1([e_i,e_j])=0.\]
	Therefore the equations above become:
	\begin{align*}
	-dh_1&=0\\
	d_{\g}h_1 +dh_2 &= 0\\
	d_{\g}h_2&=0.
	\end{align*}
	Note that the last equation is satisfied for any $h_2\in \Lambda^2{\mathfrak{so}(3)}^*\otimes C^{\infty}(M)$, since  $P_{3,\mathfrak{so}(3)}=\Lambda^3\mathfrak{so}(3)$. Therefore, the equations above become:
	\begin{align*}
	-dh_1&=0\\
	d_{\g}h_1 +dh_2 &= 0.
	\end{align*}
	It is easy to see that a trivial map $h_1\equiv 0, h_2\equiv 0$ satisfies this equation, and therefore, $\widetilde{ f}=(\widetilde{ f}_1=\widehat{f}_1, \widetilde{ f}_2\equiv 0)$ is a homotopy moment map for this action.
	
\end{ex}

\begin{remark} According to \cite[Thm. 9.6]{FRZ},
	if $G$ acts on a connected $n$-plectic manifold $(M,\omega)$ for which $H^i(M)=0$ for $1\leq i\leq n-1$, and \[[c_m^{\g}]=0\] for $c_m^{\g}$ defined in Remark \ref{classc}, then there exists a homotopy moment map for this action. In the above example $\phi\equiv 0$, hence, by Lemma \ref{phiandc}, the assumptions of the theorem stated above are satisfied, and a homotopy moment map exists. Thus, the example above is consistent with \cite[Thm. 9.6]{FRZ}.
\end{remark}

\subsection{Strict extensions} The fact that, assuming $\phi\equiv 0$, the existence of weak moment maps implies the existence of homotopy moment maps raises the following question: {Given a weak moment map and assuming $\phi\equiv0$, does there always exist a homotopy moment map that restricts to the given weak moment map?} 
The following proposition answers this question:

\begin{prop}\label{extension}
Let $\widehat{ f}$ be a weak moment map, and $\phi=0$. There exists a well-defined class $[\gamma]_{\widetilde{d}_{tot}}\in H^{n+1}(\widetilde{C})$ such that the following are equivalent:
\begin{enumerate}
\item $[\gamma]_{\widetilde{d}_{tot}}=0$ and $\gamma$ admits a primitive in 
$\bigoplus_{k=1}^{n} d_\g(\Lambda^{k}\g^*)\otimes \Omega^{n-k-1}(M)$
\item There exists a homotopy moment map $\widetilde f$, such that $\widetilde f|_{P_\g}=\widehat f$.
\end{enumerate}
\end{prop}
\begin{proof}
Let $\widehat \alpha=\widehat \alpha_1+...+\widehat \alpha_n\in \widehat C$ be a potential of $\widehat{\omega}\in \widehat C$ corresponding to $\widehat{ f}$ under the bijection in Proposition \ref{prop:weakexist}. Let $\beta \in \widetilde C$ be any preimage of $\widehat \alpha$ under the map \eqref{chainmap}. Such a preimage exists as the restriction $\widetilde{C}\to \widehat{C}$ is surjective. However, it might not be a potential of $\widetilde{ \omega}$. We can thus consider the element $\gamma=\widetilde{\omega}- \widetilde d_{tot} \beta\in \widetilde C$. Note that $\gamma \in ker(res\otimes id)=d_\g(\Lambda^{\geq1}\g^*)\otimes \Omega(M).$

First of all, we will show that $\gamma$ admitting a primitive in $d_{\g}(\Lambda^i\g)\otimes \Omega^{n-i}(M)$ does not depend on the choice of $\beta$. Indeed, if $\gamma=\widetilde{ \omega}-\widetilde{d}_{tot}\beta=\widetilde{d}_{tot}\mu$ for $\mu\in d_\g(\Lambda^{\geq1}\g^*)\otimes \Omega(M)$, then choosing a $\beta'=\beta+b$ yields $\gamma'=\widetilde{\omega}-\widetilde{d}_{tot}(\beta')=\widetilde{\omega}-\widetilde{d}_{tot}(\beta + b)=\widetilde{d}_{tot}\mu-\widetilde{d}_{tot}b=\widetilde{d}_{tot}(\mu - b)$. Now note that $b\in\ker(res\otimes id)=d_\g(\Lambda^{\geq1}\g^*)\otimes \Omega(M).$

Now, let's assume there exists a $\beta$ that corresponds to a homotopy moment map restricting to $\widehat{a}$.
Then $\gamma=\widetilde{\omega}-\widetilde{d}_{tot}\beta=0$. Choosing another $\beta'=\beta+b$, we have $\gamma=\widetilde{\omega}-\widetilde{d}_{tot}\beta'=\widetilde{\omega}-\widetilde{d}_{tot}(\beta+b)=-\widetilde{d}_{tot}b$. Again, note that $b\in\ker(res\otimes id)=d_\g(\Lambda^{\geq1}\g^*)\otimes \Omega(M).$

Conversely, assume that for some $\beta$, $\gamma=\widetilde{\omega}-\widetilde{d}_{tot}\beta=\widetilde{d}_{tot}\mu$ for $\mu\in d_\g(\Lambda^{\geq1}\g^*)\otimes \Omega(M)$. Then, $\widetilde{\omega}=\widetilde{d}_{tot}(\beta + \mu)$, i.e., $\beta+\mu$ corresponds to a homotopy moment map that restricts to $\widehat{ f}$, since $\mu\in ker(res\otimes id)$.

\end{proof}

\begin{remark}\label{gamman}
{Denote by $\gamma_{i+1}$ the component of $\gamma$ in $d_{\g}\Lambda^i\g^*\otimes \Omega^{n-i}(M)$. Note that, since $\widetilde{d}_{tot}=1\otimes d$ on $d_\g(\Lambda^{\geq1}\g^*)\otimes \Omega(M)$, it follows from $d_{tot}\gamma=0$ that $d\gamma_i=0$ for all $\gamma_i$. Requiring $\gamma$ to have a primitive $\mu\in d_\g(\Lambda^{\geq1}\g^*)\otimes \Omega(M)$ is equivalent to saying that each $\gamma_{i+1}=d\eta_i$, where $\eta_i\in d_{\g}\Lambda^i\g^*\otimes \Omega^{n-i-1}(M)$.}
	
	Indeed, suppose $\gamma=\widetilde{d}_{tot}\mu$, where $\mu\in d_\g(\Lambda^{\geq1}\g^*)\otimes \Omega(M)$. Denote the component of $\mu$ in $d_{\g}\Lambda^i\g^{*}\otimes \Omega^{n-i-1}(M)$ by $\mu_i$. Then $\gamma_{i+1}=(-1)^{i+1}d\mu_i.$
	
	Conversely, let each $\gamma_{i+1}$ satisfy $\gamma_{i+1}=d\eta_i$ for some $\eta_i\in d_{\g}\Lambda^i\g^*\otimes \Omega^{n-i-1}(M).$ Then $\widetilde{d}_{tot}(\sum\limits_{i}(-1)^{i+1}\eta_i)=\gamma$.
	
{Also note that $\gamma_{n+1}\in d_{\g}\Lambda^n\g^*\otimes C^{\infty}(M)$, and therefore $\gamma_{n+1}=d\eta$ 
if and only if $\gamma_{n+1}=0$.}
\end{remark}

\begin{cor} \label{hom}
{Let $G$ act on an n-plectic manifold $M$, let $\widehat{ f}$ be a weak moment map for this action, and let $\gamma$ be defined as in Proposition \ref{extension}. If $H^i(M)=0$ for $i\in \{1,...,n-1\}$, then there exists a homotopy moment map restricting to $\widehat{ f}$ if and only if $\gamma_{n+1}=0$.}
\end{cor}

\begin{defi}
Let $\widehat{ f}$ be a weak moment map for an action of $G$ on $(M,\omega)$. A homotopy moment map $\widetilde{ f}$ is called a \emph{strict extension} of $\widehat{ f}$ if $\widetilde{ f}|_{P_{\g}}=\widehat{ f}$, i.e., if $\widetilde{ f}$ restricts to $\widehat{ f}.$
\end{defi}

\begin{ex}
For $n=1$, i.e., in symplectic geometry, $P_{1,\g}=\g.$ Therefore, if a given weak moment map is not already a homotopy moment map, there is no homotopy moment map restricting to it.

To see this in terms of the results of Prop. \ref{extension}, let $\widehat{ f}$ be a symplectic weak moment map. {Note that in this case $\gamma=\gamma_2$, and so by the Remark \ref{gamman} and Proposition \ref{extension}, there exists a homotopy moment map restricting to $\widehat{ f}$ if and only 	
\begin{align*}
\gamma(x,y) &=-\omega_2(x,y) -d_{\g}\widehat{f}(x,y)\\
&=-\omega(v_x,v_y)+\widehat{ f}([x,y])\\
\end{align*} 
vanishes, i.e., if and only if $\widehat{ f}$ is already an equivariant moment map, i.e., a homotopy moment map.}

{For an example of a symplectic weak moment map that cannot be strictly extended to a homotopy moment map, consider a Lie algebra $\g$ such that $H^1(\g)=0$. If there exists a homotopy moment map $\widetilde{f}$ for the action of $\g$, then it is unique (see, e.g., \cite[\S 26]{Ana}). On the other hand, for an arbitrary nonzero $\xi \in \g^*$ the map $\widehat{f}:=\widetilde{f}+\xi$,  satisfies the condition 
	$d\widehat{f}(x)=-\iota_{v_x}\omega$
	for all $x\in \g$, i.e., is a weak moment map, but not a homotopy moment map.}
\end{ex}

\begin{ex}Consider the homotopy moment map $\widetilde{ f}$ constructed in Example \ref{so(3)} for the action of $SO(3)$ on $(\RR^3, \omega=dx_1\wedge dx_2\wedge dx_3)$, given by 
	$\widetilde{f}_1(e_1)=\omega(v_2,v_3)$,
	$\widetilde{f}_1(e_2)=-\omega(v_1,v_3)$,
	$\widetilde{f}_1(e_3)=\omega(v_1,v_2)$
and $\widetilde{f}_2\equiv 0$.
This homotopy moment map coincided with the original weak moment map, i.e., in this case the original weak moment map admitted an obvious "strict extension" to a homotopy map. To see this in the context of Proposition \ref{extension}, note that in this case 
\begin{align*}
\gamma&=\omega_1 - \omega_2 + \omega_3 - d_{\g}\widetilde{f}_1 + d\widetilde{ f}_1\\
&=\omega_3\\
&=0,
\end{align*}
since in this case $\omega_2(e_i,e_j) {-}\widetilde{ f}_1([e_i,e_j])=0$ and $\omega_3=\phi=0$ (see the discussion in Example \ref{so(3)}),
i.e., {$\gamma=\gamma_3$ vanishes in this example, and by Remark \ref{gamman} there indeed exists a strict extension of the weak moment map to a homotopy moment map.}

{Moreover, any weak moment map for this action can be strictly extended to a homotopy moment map. Indeed, let $\widehat{ f}=(\widehat{ f}_1,\widehat{ f}_2\equiv 0)$ be a weak moment map for this action. The equations for a homotopy moment map are 
\begin{align*}
d\widetilde{ f}_1&=-\iota_{v_x}\omega\\
\widetilde{ f}_1([x,y])&=d\widetilde{ f}_2(x,y)+\omega(v_x,v_y)\\
-\widetilde{ f}_2(\delta(x,y,z))&=-\omega(v_x,v_y,v_z).
\end{align*}
Note that any $\widetilde{ f}_2\in \Lambda^2\mathfrak{so}(3)^*\otimes C^{\infty}(M)$ restricts to $\widehat{f}_2\equiv 0$ and satisfies the third equation above, since $P_{2,\mathfrak{so}(3)}=0$, $P_{3,\mathfrak{so}(3)}=\Lambda^3\mathfrak{so}(3)$, and $\omega_3\equiv0$. 
Also note that, if the first one of the above equations is satisfied, then $d\widetilde{ f}_1([x,y])=-\iota_{[v_x,v_y]}=d\omega(v_x,v_y)$, and therefore the difference $\widetilde{ f}_1([x,y])-\omega(v_x,v_y)$ is a closed 1-form on $\RR^3$ for all $x,y\in \mathfrak{so}(3)$. Since $H^1(\RR^3)=0$, this form is exact, and there exists a $\widetilde{f}_2$ satisfying the second of the equations above. Therefore, $\widetilde{ f}_1=(\widehat{ f}_1, \widetilde{ f}_2)$, is a homotopy moment map that restricts to the given weak moment map $\widehat{ f}=(\widehat{ f}_1,\widehat{ f}_2\equiv 0)$. {This result is consistent with the Corollary \ref{hom} and the fact that $\gamma_3=\omega_3 - d_{\g}\widetilde{f}_2=0$, since both $\omega_3$ and $d_{\g}\widetilde{f}_2$ vanish.}}

\end{ex}



\section{Equivariance}

\begin{defi}[\cite{FRZ,MR3815227}] \label{equi} Let $G$ be a Lie group acting on $(M,\omega)$ preserving $\omega$.
	A homotopy moment map $f: \g \to L_{\infty}(M,\omega)$ is called \emph{equivariant} if for all $g \in G, p\in \Lambda^k \g,$ and $1\leq k \leq n$ 
	\begin{equation} \label{equiv}
	f_k(Ad_{g}p)=\Phi^*_{g}f_k(p),	
	\end{equation}
	where $\Phi^*_{g}$ denotes the pullback action. It is \emph{infinitesimally equivariant} or \emph{$\g$-equivariant} if and only if for all $x\in \g, p\in \Lambda^k\g$  and $1 \leq k \leq n$  
	\begin{equation}\label{infinitequiv}
	f_k(ad_xp)-\pounds_{v_x}f_k(p)=0,
	\end{equation}
	where $ad$ denotes the adjoint action of $\g$ on $\Lambda^k\g$. In complete analogy, a weak homotopy moment map is \emph{equivariant} if \eqref{equiv} holds for all $p\in P_{k,\g}$ resp. \emph{infinitesimally equivariant} if \eqref{infinitequiv} holds for all $x\in \g, p\in P_{k,\g}$ and $1 \leq k \leq n$.
\end{defi}

\begin{remark}
	For a connected Lie group $G$, a homotopy (or weak) moment map is equivariant if and only if it is infinitesimally equivariant. We will treat the case of infinitesimal equivariance in the sequel, the equivariant working in complete analogy.
\end{remark}

Consider the complex $\widetilde C^\g= (\Lambda^{\geq 1}\g^* \otimes \Omega(M))^{\g}$, consisting of all $\g$-invariant elements of \eqref{big complex}. The total differential \eqref{dtot} restricts to $\widetilde C^\g$, because $d_{\g}$ is $\g$-equivariant and $d$ commutes with the Lie derivative. Since the adjoint action $ad:\g\to  End(\Lambda \g)$ preserves the subspace of $\delta$-closed elements, it defines an action on $P_{\g}$ and thus on $\widehat C=P^*_{\g} \otimes \Omega(M)$. Again, the total differential $\widehat d_{tot}$ restricts to a differential on $\widehat C^\g$.

%

\begin{lemma} \label{omegainvariant}
	The element $\widetilde{\omega}$ (resp. $\widehat \omega$) lies in $\widetilde{C}^\g$ (resp. $\widehat{C}^\g$).
\end{lemma}
\begin{proof}
	We will prove the statement for $\widetilde{\omega}$, the statement for $\widehat \omega$, follows as the image of a $\g$-invariant element under the equivariant map $(res\otimes id)$ is necessarily $\g$-invariant. I.e., we have to show that for $k\in\{1,...,n+1\}$ $\omega_k$ is $\g$-invariant. We calculate
	\begin{align*}
	 \pounds_{v_x}\omega_k(p)=d\iota_{v_x}\iota_{v_p}\omega=d\iota_{p\wedge x}\omega
	 =(-1)^{k}d\iota_{x\wedge p}\omega\overset{(a)}=-\iota_{\delta(x\wedge p)}\omega\overset{(b)}=\iota_{ad_xp}\omega=\omega_k(ad_xp),
	 \end{align*}
	 where we used Lemma \ref{cartan} for equality (a) and Cartan's magic formula $ ad_x(p)=\delta(x\wedge p)-x\wedge \delta(p)$ for equality (b).	
\end{proof}

The correspondences between potentials and moment maps established in Propositions \ref{prop:hotexists} and \ref{prop:weakexist} carry over to the $\g$-equivariant setting and we have the following
 
\begin{thm}
	Let $\g$ act on $(M,\omega)$ by preserving $\omega$. The action admits
	\begin{enumerate}
		\item a $\g$-equivariant weak moment map if and only if $[\widehat\omega]=0\in H(\widehat C^\g)$
		\item a $\g$-equivariant homotopy moment map if and only if $[\widetilde\omega]=0\in H(\widetilde C^g)$
	\end{enumerate}
	Moreover, the respective moments are in one-to-one correspondence with potentials of the respective cohomology classes.
\end{thm}

This theorem recovers the following result:
\begin{cor}[Proposition 7.3 and Theorem 7.4 in  \cite{MR3815227}]
	$\g$-equivariant weak moment maps are unique up to elements of $\bigoplus_{k=1}^n\left(P^*_{k,g}\otimes \Omega^{n-k+1}_{cl}(M)\right)^\g$. In particular, if these groups vanish, then $\g$-equivariant weak moment maps are unique.
\end{cor}

\subsection{Equivariantization}
In this subsection, we will determine the condition that determines when homotopy and weak moment maps can be made equivariant. We review the Theorem 4.10 from \cite{MR3815227} in the terms of this paper and derive its analogue for homotopy moment maps. To do this, we need some Lie algebra cohomology:\\

Let $\mathcal M$ be a $\g$-module and $\mathcal N$ a $\g$-submodule. Let $\alpha$ in $\mathcal M$ be an element satisfying $x\cdot \alpha\in \mathcal N$ for all $x \in \g$. We can ask ourselves, whether $\alpha$ can be changed by an element $\beta$ in $\mathcal N$ such that $(\alpha-\beta)\in \mathcal M^\mathfrak g$. This is equivalent to finding $\beta \in \mathcal N$, such that $x\cdot\alpha=x\cdot \beta$ for all $x \in \mathfrak g$. The map $x\mapsto x\cdot \alpha$ interpreted as the element $d_{g,\mathcal M}(\alpha)$ in $\mathfrak g^ *\otimes \mathcal N$ is closed with respect to the Lie algebra cohomology differential $d_{\g, \mathcal N}$, and potentials correspond to elements $\beta\in \mathcal N$ such that $d_{\g,\mathcal M}(\alpha)=d_{\g,\mathcal N}(\beta)$.\\

Now, take $\mathcal M=\widehat C$ and  $\alpha\in \widehat C$ a potential of $\widehat{\omega}$ corresponding to some weak moment map. Then $d_{\g,\mathcal M}(\alpha)\in \mathcal N$, where $\mathcal N=P_{\geq 1,\g}^*\otimes \Omega_{cl}(M)$ is the space of admissible changes for weak moment maps (cf. Proposition \ref{prop:weakexist}). The above discussion implies: 

\begin{thm}[Theorem 4.10 in \cite{MR3815227}] Let $\widehat{\alpha}\in \widehat C$ be a potential of $\widehat{\omega}$ corresponding to a weak moment map $\widehat{ f}$. Then $d_{\g,\widehat C}(\alpha)$ yields a well-defined cohomology class in $H^1(\mathfrak g, P_{\geq 1,\g}^*\otimes \Omega_{cl}(M))$. Furthermore, $[d_{\g,\widehat C}(\alpha)]\in H^1(\mathfrak g, P_{\geq 1,\g}^*\otimes \Omega_{cl}(M))$ vanishes if and only if there exists a $\g$-equivariant weak moment map.
\end{thm}

An analogous construction works for homotopy moment maps:

\begin{prop} Let $\widetilde{\alpha}\in \widetilde C$ be a potential of $\widetilde {\omega}$ corresponding to a homotopy moment map $\widetilde{ f}$. Then $\widetilde C(\alpha)$ yields a well-defined cohomology class in $H^1(\mathfrak g,\mathcal N)$, where $\mathcal N$ is the subspace of $\widetilde d_{tot}$-closed elements in $\Lambda^{k\geq 1}\g\otimes \Omega(M)$. Furthermore, $[d_{\g, \mathcal N}(\alpha)]\in H^1(\mathfrak g,\mathcal N)$ vanishes if and only if there exists a $\g$-equivariant homotopy moment map.
\end{prop}

\begin{remark}
We end the paper by noting that, in symplectic geometry, considering actions of connected Lie groups, the conditions for being a homotopy moment map and an equivariant moment map coincide, i.e., a symplectic homotopy moment map is automatically equivariant.
\end{remark}

%
%

\bibliographystyle{habbrv} 
\bibliography{Lie2Alg2019}

\end{document}